\documentclass[11pt]{article}
\usepackage[a4paper, total={17cm,25cm}]{geometry}

\usepackage{hyperref}

\usepackage{amsmath,amssymb,amsfonts}
\usepackage{color, graphics, soul}
\usepackage{tikz}
\usepackage{float}
\usepackage[normalem]{ulem}

\usepackage[alphabetic,y2k,initials]{amsrefs}

\def\C{\mathcal{C}}

\def\E{\mathcal{E}}

\def\E{\mathcal{E}}

\numberwithin{equation}{section}

\newtheorem{theorem}{Theorem}[section]
\newtheorem{proposition}[theorem]{Proposition}
\newtheorem{corollary}[theorem]{Corollary}
\newtheorem{lemma}[theorem]{Lemma}
\newtheorem{remark}[theorem]{Remark}

\newtheorem{definition}[theorem]{Definition}

\newenvironment{proof}[1][Proof]{\noindent\textit{#1.} }{\hfill$\Box$\medskip}

 \title{Isoperiodic families of Poncelet polygons inscribed in a circle and circumscribed about conics from a confocal pencil}

\usepackage{authblk}
\author[1,3]{Vladimir Dragovi\'c}
\author[2,3]{Milena Radnovi\'c}
\affil[1]{\textsc{The University of Texas at Dallas, Department of Mathematical Sciences}}
\affil[2]{\textsc{The University of Sydney, School of Mathematics and Statistics}}
\affil[3]{\textsc{Mathematical Institute SANU, Belgrade}}
\affil[ ]{\texttt{vladimir.dragovic@utdallas.edu, milena.radnovic@sydney.edu.au}}

\date{}

\begin{document}

\maketitle

\begin{abstract}

Poncelet polygons inscribed in a circle and circumscribed about conics from a confocal family naturally arise in the analysis of the numerical range and Blaschke products.
We examine the behaviour of such polygons when the inscribed conic varies through a confocal pencil and discover cases when each conic from the confocal family is inscribed in an $n$-polygon, which is inscribed in the circle, with the same $n$.
Complete geometric
characterization of such cases for $n\in\{4,6\}$ is given and
proved that this cannot happen for other values of $n$.
We establish  a
relationship of such families of Poncelet quadrangles and hexagons to solutions of a Painlev\'e VI equation.


\emph{AMS subclass:} 14H70, 37J70, 34M55, 37A10
\end{abstract}
\newpage
\tableofcontents

\section{Introduction}\label{sec:intro}

Although Poncelet polygons have been occupying the attention of major mathematicians for more than a bicentennial, from Poncelet, Jacobi, Cayley, via Darboux and Lebesgue,to Griffiths, Harris, Berger, Kozlov, Narasimhan, and Hitchin, to mention a few, it seems that new, interesting and unexpected phenomena and connections related to Poncelet porism are still waiting to be discovered.

The main object of our paper are
Poncelet polygons inscribed in a given circle and circumscribed about conics from a confocal family,  which is a setting that emerges from the study of the numerical range and Blaschke products, see \cite{DGSV} and \cite{MFSS2019} and references therein.	

In this work, we found that in some cases all such Poncelet polygons have the same number of sides, even as the inscribed conic changes in the confocal pencil.
This new phenomenon substantially improves and modifies the understanding  and intuition related to injectivity and  monotonicity of induced rotational numbers. (More about rotational numbers can be found in e.g. \cites{King,Duis, DR2011knjiga, DragRadn2014jmd, DR2019cmp}.)
We give elegant classical geometric characterisations of those families for $n=4$ and $n=6$ and show that no other natural $n$ gives rise to such families.

We also discern both analogy and disparity of the obtained results with the case when the circumscribed conic is not a circle but an ellipse  which also belongs to the confocal family.
  That case has been extensively studied in connection with the elliptical billiards and the Great Poncelet Theorem, see for example \cites{BergerGeometryII,DarbouxSUR,KozTrBIL,LebCONIQUES,FlattoBOOK,DR2011knjiga,DragRadn2014bul,DR2019rcd} and references therein.

This paper is organised as follows.
In Section \ref{sec:Poncelet}, we review the Poncelet porism and analytic conditions for closure in our setting.
In particular, we analyse those conditions for polygons with three, four and five sides.

At the beginning of Section \ref{sec:families46}, we introduce \emph{isoperiodic families}.
Each such family consists of a circle and a confocal family of conics, such that each polygon inscribed in the circle and circumscribed about any of the confocal conics is closed with $n$ sides.
Theorems \ref{th:isorotational4} and \ref{th:isorotational6} give a complete characterisation such isoperiodic families for $n\in\{4,6\}$.
Section \ref{sec:noisorot} addresses the natural question of the existence of such families exist for other values of $n$ and gives the negative answer in Theorem \ref{th:noisorot}.

Section \ref{sec:blaschke} provides a further study of Blaschke ellipses. In particular, in Theorem \ref{th:Bla3} we characterize the Blaschke ellipse as the only conic from a confocal pencil, \emph{with both focal points in the interior of the unit disk}, which is $3$-Poncelet inscribed in the unit circle.
This theorem also answers negatively the following question:

\emph{Is it possible to have a triangle inscribed in a circle and circumscribed about an ellipse intersecting the circle whose foci are within the circle?}

Similarly, in Theorem \ref{th:Bla4} we characterize the Blaschke ellipse as the only conic from a confocal pencil, \emph{with both focal points in the interior of the unit disk}, which is $4$-Poncelet inscribed in the unit circle.

In Section \ref{sec:painleve}, the obtained isoperiodic confocal families are used for the construction of explicit solutions to a Painleve\'e VI equation, see Theorem \ref{th:painleve} for $n=4$ and Theorem \ref{th:painlevek6} and for $n=6$.

\section{Poncelet porism and Cayley's conditions}\label{sec:Poncelet}

Given two smooth conics in a general position in the plane, the Poncelet porism states that that existence of a closed polygonal line inscribed in one of those conics and circumscribed about the other one will imply the existence of infinitely many such polygons, which will all have the same number of sides, see Figure \ref{fig:hepta}.
 \begin{figure}[h]
	\begin{center}
		\begin{tikzpicture}[scale=4]
			\draw[thick](0.25,0.15) circle (1);
			\draw[thick](0,0) ellipse (0.595511 and 0.514522);
			
			\draw[very thick] (0.772687, 1.00252)--(0.460977, -0.827491)--(-0.660792, -0.262865)--(-0.41634, 0.895648)--(1.21203, -0.122948)--(-0.320089, -0.671583)--(-0.748112, 0.211419)--(0.772687, 1.00252);
			\draw[black,fill=black](0.772687, 1.00252) circle (0.02);
			\draw[black,fill=black](0.460977, -0.827491) circle (0.02);
			\draw[black,fill=black](-0.660792, -0.262865) circle (0.02);
			\draw[black,fill=black](-0.41634, 0.895648) circle (0.02);
			\draw[black,fill=black](1.21203, -0.122948) circle (0.02);
			\draw[black,fill=black](-0.320089, -0.671583) circle (0.02);
			\draw[black,fill=black](-0.748112, 0.211419) circle (0.02);
			
			\draw[very thick, gray](1.00075,0.810584)--(0.258419,-0.849965)--(-0.697181,-0.170699)--(-0.22,1.03267)--(1.1145,-0.352637)--(-0.425556,-0.587309)--(-0.731608,0.340906)--(1.00075,0.810584);
				\draw[gray,fill=gray](1.00075,0.810584) circle (0.02);
			\draw[gray,fill=gray](0.258419,-0.849965) circle (0.02);
			\draw[gray,fill=gray](-0.697181,-0.170699) circle (0.02);
			\draw[gray,fill=gray](-0.22,1.03267) circle (0.02);
			\draw[gray,fill=gray](1.1145,-0.352637) circle (0.02);
			\draw[gray,fill=gray](-0.425556,-0.587309) circle (0.02);
			\draw[gray,fill=gray](-0.731608,0.340906) circle (0.02);

		\end{tikzpicture}
		\caption{Two heptagons inscribed in a circle and circumscribed about an ellipse.}\label{fig:hepta}
	\end{center}
\end{figure}
Moreover, there will be a polygon with $n$ sides inscribed in conic $\Gamma$ and circumscribed about conic $\mathcal{C}$ if and only if:
\begin{gather*}
	C_2=0,
	\quad
	\left|
	\begin{array}{cc}
		C_2 & C_3
		\\
		C_3 & C_4
	\end{array}
	\right|=0,
	\quad
	\left|
	\begin{array}{ccc}
		C_2 & C_3 & C_4
		\\
		C_3 & C_4 & C_5
		\\
		C_4 & C_5 & C_6
	\end{array}
	\right|=0,
	\dots
	\quad\text{for}\quad n=3,5,7,\dots
	\\
	C_3=0,
	\quad
	\left|
	\begin{array}{cc}
		C_3 & C_4
		\\
		C_4 & C_5
	\end{array}
	\right|=0,
	\quad
	\left|
	\begin{array}{ccc}
		C_3 & C_4 & C_5
		\\
		C_4 & C_5 & C_6
		\\
		C_5 & C_6 & C_7
	\end{array}
	\right|=0,
	\dots
	\quad\text{for}\quad n=4,6,8,\dots,
\end{gather*}
where $\sqrt{\det(\mathcal{C}+\lambda\Gamma)}
=
C_0+C_1\lambda+C_2\lambda^2+C_3\lambda^3+\dots.
$
is the Taylor expansion about $\lambda=0$, and we used letters $\mathcal{C}$ and $\Gamma$ to denote the matrices of those two conics in the projective plane.
This is also equivalent to the divisor relation $n(P_0-P_{\infty})\sim0$ on the elliptic curve
$\mu^2=\det(\C+\lambda\Gamma)$.
The analytic condition for periodicity was derived by Cayley \cite{Cayley1853}, see also \cites{GrifHar1978,DR2011knjiga,DR2019rcd} and references therein.

\begin{remark}
Two non-degenerate conics are in a general position when their complexifications intersect at four distinct points.
In that case, the cubic curve $\mu^2=\det(\C+\lambda\Gamma)$ is smooth.
The conics have a point of tangency if and only if the polynomial $\det(\C+\lambda\Gamma)$ has a multiple root, meaning that the curve has a singular point.
\end{remark}

\begin{definition}\label{def:n-pair}
If there is a polygon with $n$ sides inscribed in $\Gamma$ and circumscribed about $\C$, we will say that $\Gamma$ and $\C$ are an \emph{$n$-Poncelet pair} or that $\mathcal{C}$ is \emph{$n$-Poncelet inscribed in} $\Gamma$.
\end{definition}

In this work, we will consider the case when $\Gamma$ is a circle:
$$
\Gamma:\ (x-x_0)^2+(y-y_0)^2=1
$$
and $\mathcal{C}=\mathcal{C}(t)$ a conic from the following confocal family:
$$
\mathcal{C}(t):\ \frac{x^2}{a-t}+\frac{y^2}{b-t}=1, \quad a>b>0.
$$
As above, we will use the same letters to denote the matrices corresponding to those two conics:
$$
\Gamma
=
\left(
\begin{array}{ccc}
1 & 0 & -x_0\\
0 & 1 & -y_0\\
-x_0 & -y_0 & x_0^2+y_0^2-1
\end{array}
\right),
\quad
\mathcal{C}(t)=
\left(
\begin{array}{ccc}
\frac1{a-t} & 0 & 0\\
0 & \frac1{b-t} & 0\\
0 & 0 & -1
\end{array}
\right).
$$
We denote:
\begin{equation}\label{eq:Dt}
\begin{aligned}
D_t(\lambda)=&\ (a-t)(b-t)\det(\mathcal{C}(t)+\lambda\Gamma)
\\
=&-(a-t) (b-t) \lambda ^3
+\lambda ^2
\left(
(a-t) (y_0^2-1)+(b-t) (x_0^2-1)-(a-t) (b-t)
\right)
\\&
+\lambda  \left(x_0^2+y_0^2-1-(a-t)-(b-t)\right)
-1,
\end{aligned}
\end{equation}
and the Taylor series around $\lambda=0$:
$$
\sqrt{D_t(\lambda)}
=
C_0(t)+C_1(t)\lambda+C_2(t)\lambda^2+C_3(t)\lambda^3+\dots.
$$
Notice that the coefficients $C_k(t)$ are polynomial in $t$, as well as in $a$, $b$, $x_0^2$, $y_0^2$.  Thus, the Cayley conditions:
\begin{equation}\label{eq:pnt}	
\begin{aligned}
&p_n(t):=
\det\left(
\begin{array}{llll}
	C_2(t) & C_3(t) & \dots & C_{m+1}(t)\\
	C_3(t) & C_4(t) & \dots & C_{m+2}(t)\\
	 \dots&\dots&\dots&\dots\\
	C_{m+1}(t)& C_{m+2}(t) & \dots & C_{2m}(t)
\end{array}
\right) = 0,
&\quad\text{for}&\quad n=2m+1,
\\
&p_n(t):=
\det\left(
\begin{array}{llll}
	C_3(t) & C_4(t) & \dots & C_{m+2}(t)\\
	C_4(t) & C_5(t) & \dots & C_{m+3}(t)\\
	\dots&\dots&\dots&\dots\\
	C_{m+2}(t)& C_{m+3}(t) & \dots & C_{2m+1}(t)
\end{array}
\right) = 0,
&\quad\text{for}&\quad n=2m+2.
\end{aligned}
\end{equation}
 define polynomials in $t$, denoted as $p_n(t)$, which are also polynomials in $a$, $b$, $x_0^2$, $y_0^2$.

Let us denote the corresponding family of curves as:
\begin{equation}\label{eq:elliptic-curves}
\mathcal{E}_t: \mu^2=D_t(\lambda).
\end{equation}

\begin{remark}
The curve $\mathcal{E}_t$ is elliptic whenever conic $\C(t)$ is non-degenerate and with no points of tangency with $\Gamma$, i.e.~whenever $t\not\in\{a,b\}$ and all roots of $D_t(\lambda)$ are simple.
Thus, the curve $\mathcal{E}_t$ is smooth for all but finitely many values of $t\in\mathbf{C}$.

Notice that, for $t\in\{a,b\}$, the curve $\E_t$ is not cubic: its degree will be equal $2$.
\end{remark}

\begin{proposition}\label{prop:npair}
Circle $\Gamma$ and conic $\mathcal{C}(t)$ are an $n$-Poncelet pair if and only if
$$
		p_n(t)=0,
$$
	for smooth and non-smooth cubics  $\mathcal{E}_t$, i.e.~for any $t\not\in\{a,b\}$,  where $p_n(t)$ is defined in \eqref{eq:pnt}.
\end{proposition}
\begin{proof}
If $\mathcal{E}_t$ is smooth, this is just a straightforward Cayley's condition.
Otherwise, $D_t(\lambda)$ is a cubic polynomial with a double root, which is equivalent to conics $\Gamma$ and $\mathcal{C}(t)$ having at least one point of tangency.
In this case, the conclusion follows from the analysis in \cite{FlattoBOOK}  (see also \cite{DR2024rcd}).
\end{proof}

Now, we are going to expand the above Cayley's conditions for a few small values of $n$.

\begin{corollary}\label{cor:k3}
Conic $\mathcal{C}(t)$ is $3$-Poncelet inscribed in $\Gamma$ if and only if $C_2(t)=0$:
$$
4t+1 - 2 (a+b) + (a-b)^2
  - 2 x_0^2(1+a-b)
   - 2 y_0^2(1-a+b) + (x_0^2+y_0^2)^2=0.
$$
That equation has the unique solution in $t$, which we will denote as $t_3$:
\begin{equation}\label{eq:t3}
t_3=-\frac{1}{4} \left(1 - 2 (a+b) + (a-b)^2
- 2 x_0^2(1+a-b)
- 2 y_0^2(1-a+b) + (x_0^2+y_0^2)^2\right).
\end{equation}
\end{corollary}

\begin{corollary}\label{cor:k4}
Conic $\mathcal{C}(t)$ is $4$-Poncelet inscribed in $\Gamma$ if and only if $C_3(t)=0$, i.e:
$$
\alpha_4+\beta_4t=0,
$$
where:
\begin{gather*}
\alpha_4=8ab+4(x_0^2+y_0^2-a-b-1)t_3,
\\
\beta_4=8t_3+4(1-a-b-x_0^2-y_0^2),
\end{gather*}
and $t_3$ is given by \eqref{eq:t3}.
\end{corollary}

\begin{corollary}\label{cor:k5}
Conic $\mathcal{C}(t)$ is $5$-Poncelet inscribed in $\Gamma$ if and only if:
$$
\det\left(
\begin{array}{cc}
C_2(t) & C_3(t)\\
C_3(t) & C_4(t)
\end{array}
\right)
=0,
$$
which is equivalent to:
$$
\alpha_5+\beta_5t+\gamma_5t^2+\delta_5t^3=0,
$$
with
$$
\begin{aligned}
\alpha_5=\ &
-64 t_3^3+
128ab \left(  a +  b+ 1-  x_0^2-y_0^2\right)t_3-256 a^2 b^2,
\\
\beta_5=\ &
192t_3^2
+128 t_3
\left(
(a+b) (x_0^2+ y_0^2-a-b-1)
-2 a b
\right)
+128 a b \left(3 a+3 b+x_0^2+y_0^2-1\right),
\\
\gamma_5
=\ &
64 \left(6 a+6 b-2 x_0^2-2 y_0^2-1\right)t_3
-128 \left(a^2+4 a b+a x_0^2+a y_0^2-a+b^2+b x_0^2+b y_0^2-b\right),
\\
\delta_5
=\ &-256t_3+64 \left(2 a+2 b+2 x_0^2+2 y_0^2-1\right),
\end{aligned}
$$
and $t_3$ is given by \eqref{eq:t3}.
\end{corollary}

\section{Isoperiodic families of Poncelet polygons}
\label{sec:families46}

In this section, we are searching for a circle $\Gamma$ and a confocal pencil $\C(t)$, such that $\C(t)$ is $n$-Poncelet inscribed in $\Gamma$ for each $t\not\in\{a,b\}$, with a fixed $n$.
In such a case, we will say that $(\Gamma, \C(t))$ is \emph{an isoperiodic family}.

Notice that, according to Corrollary \ref{cor:k3}, there are no isoperiodic families for $n=3$.

We will focus our search to the cases $n\in\{4,6\}$, because of the simplicity of the Cayley's conditions for those values.

\subsection{Quadrangles}\label{sec:k=4}

In order to simplify the calculations, we will start with the analysis of the cases when the centre of the circle $\Gamma$ belongs to one of the axes of the confocal family.

The first is case when the centre of $\Gamma$ is on the $y$-axis, i.e.~$x_0=0$.

The condition for quadrangles, $C_3(t)=0$, according to Corollary \ref{cor:k4}, becomes:
$$
\left(a-b+y_0^2-1\right)
\left(
2 \left(a-b+y_0^2+1\right) t-
(a^2-b^2+2 b y_0^2+2 b-y_0^4+2 y_0^2-1)
\right)
=0.
$$
Having in mind that $a>b$, which implies that the second factor in the above expression is non-constant,
we can conclude that $C_3(t)=0$ for all $t$ is equivalent to $a-b+y_0^2=1$.
It is easy to see that this is equivalent to the foci of the family $\mathcal{C}(t)$ lying on the circle $\Gamma$, see Figure \ref{fig:quadrangles}.
 \begin{figure}[h]
	\begin{center}
		\begin{tikzpicture}[scale=4]
		\draw[thick](0,0.3) circle (1);
		\draw[black,fill=gray](0.954,0) circle (0.02);
		\draw[black,fill=gray](-0.954,0) circle (0.02);
		\draw[thick,gray](0,0) ellipse (1.095 and 0.5385);
		\draw[thick,gray](0,0) ellipse (1.22474 and 0.7681151);
		
		\draw (-0.29552, 1.25534)--(0.997618, 0.231014)--(0.29552,1.25534)--(-0.997618, 0.231014)--(-0.29552, 1.25534);
		\draw[black,fill=black](-0.29552, 1.25534) circle (0.02);
		\draw[black,fill=black](0.997618, 0.231014) circle (0.02);
		\draw[black,fill=black](0.29552, 1.25534) circle (0.02);
		\draw[black,fill=black](-0.997618, 0.231014) circle (0.02);
		\draw[gray,fill=gray!50](0.930769, 0.283966) circle (0.02);
		\draw[gray,fill=gray!50](-0.930769, 0.283966) circle (0.02);
		\draw[gray,fill=gray!50](1.0381, 0.171955) circle (0.02);
		\draw[gray,fill=gray!50](-1.0381, 0.171955) circle (0.02);
		
		\draw (-0.479426, 1.17758)--(0.979586, 0.501026)--(0.479426, 1.17758)--(-0.979586, 0.501026)--(-0.479426, 1.17758);
		\draw[black,fill=black](-0.479426, 1.17758) circle (0.02);
		\draw[black,fill=black](0.979586, 0.501026) circle (0.02);
		\draw[black,fill=black](0.479426, 1.17758) circle (0.02);
		\draw[black,fill=black](-0.979586, 0.501026) circle (0.02);
		\draw[gray,fill=gray!50](0.728133, 0.617627) circle (0.02);
		\draw[gray,fill=gray!50](1.11113, 0.323094) circle (0.02);
		\draw[gray,fill=gray!50](-0.728133, 0.617627) circle (0.02);
		\draw[gray,fill=gray!50](-1.11113, 0.323094) circle (0.02);
		
		\end{tikzpicture}
		\caption{A circle and any conic with the foci on the circle are a $4$-Poncelet pair.}\label{fig:quadrangles}
	\end{center}
\end{figure}
\begin{remark}\label{rem:43}
Notice that the Corollary \ref{cor:k3} implies that there is a unique value of $t$ such that the condition for triangles is satisfied.
In the case when the foci of $\mathcal{C}(t)$ lie on $\Gamma$, this is seemingly in contradiction with the conclusion that the condition for quadrangles is satisfied for any $t$.

When we calculate $t_3$ using the formula \eqref{eq:t3} from Corollary \ref{cor:k3}, with the condition  $a-b+y_0^2=1$, we get $t_3=b$, which corresponds to the singular conic $\C(b)$ from the confocal family $\mathcal{C}(t)$.
\end{remark}

Consider now the other symmetric case: $y_0=0$.
The condition for quadrangles, $C_3(t)=0$, is:
\begin{equation}\label{c3=0}
\left(a-b-x_0^2+1\right)
\left(
2( x_0^2+1 - a + b) t
+
(a^2-b^2-2 a x_0^2-2a+x_0^4-2 x_0^2+1)
\right)
=0,
\end{equation}
so we see that this condition is satisfied for each $t$ whenever either:
\begin{itemize}
	\item $x_0^2=a-b+1$; or
	\item $- a + b + x_0^2+1
	=
	a^2-b^2-2 a x_0^2-2a+x_0^4-2 x_0^2+1
	=0$.
\end{itemize}
Notice that the second possibility leads to $y_0=0$ and $a=b+1$, which a special case represented in Figure \ref{fig:quadrangles}, when the centre of the circles coincides with the centre of the confocal family.

On the other hand, we can see the geometric meaning of the first possibilty by calculating the cross ratio of the foci with the endpoints of the diameter of $\Gamma$ which lies on the $x$-axis:
$$
\frac{\sqrt{a-b}-(x_0-1)}{\sqrt{a-b}-(x_0+1)}
\cdot
\frac{-\sqrt{a-b}-(x_0+1)}{-\sqrt{a-b}-(x_0-1)}
=
-1.
$$
Thus, we get that the foci are symmetric to each other with respect to $\Gamma$, see Figure \ref{fig:quad2}.
\begin{figure}[h]
\centering
		\begin{tikzpicture}[scale=4]
		\draw[thick](1.22474,0) circle (1);
		\draw[black,fill=gray](0.7071,0) circle (0.02);
		\draw[black,fill=gray](-0.7071,0) circle (0.02);
		\draw[thick,gray](0,0) ellipse (1 and 0.7071);
		\draw[thick,gray](0,0) ellipse (1.22474 and 1);
		
		\draw (1.22474, 1.)--(0.816497, -0.912871)--(1.22474, -1.)--(0.816497, 0.912871)--(1.22474, 1.);
		\draw[black,fill=black](1.22474, 1.) circle (0.02);
		\draw[black,fill=black](1.22474, -1.) circle (0.02);
		\draw[black,fill=black](0.816497, -0.912871) circle (0.02);
		\draw[black,fill=black](0.816497, 0.912871) circle (0.02);
		\draw[gray,fill=gray!50](0.988804, -0.105516) circle (0.02);
		\draw[gray,fill=gray!50](-0.288949, -0.676945) circle (0.02);
		\draw[gray,fill=gray!50](0.988804, 0.105516) circle (0.02);
		\draw[gray,fill=gray!50](-0.288949, 0.676945) circle (0.02);
		
		\draw (0.745319, 0.877583)--(2.01256, -0.615912)--(0.745319, -0.877583)--(2.01256, 0.615912)--(0.745319, 0.877583);
		\draw[black,fill=black](0.745319, 0.877583) circle (0.02);
		\draw[black,fill=black](2.01256, -0.615912) circle (0.02);
		\draw[black,fill=black](0.745319, -0.877583) circle (0.02);
\draw[black,fill=black](2.01256, 0.615912) circle (0.02);
		\draw[gray,fill=gray!50](1.00674, 0.569485) circle (0.02);
		\draw[gray,fill=gray!50](0.30028, -0.969478) circle (0.02);
		\draw[gray,fill=gray!50](1.00674, -0.569485) circle (0.02);
		\draw[gray,fill=gray!50](0.30028, 0.969478) circle (0.02);
		
		\end{tikzpicture}
		\caption{Each conic from the confocal family is inscribed in quadrangles whose vertices belong to the circle. The foci are symmetric to each other with respect to the circle.}\label{fig:quad2}
\end{figure}

\begin{remark}
Similarly as in Remark \ref{rem:43}, we should check the case when both Cayley's conditions, for triangles and quadrangles, are satisfied: $C_2(t)=C_3(t)=0$.
The value of $t_3$ from Corollary \ref{cor:k3} is $t_3=a$ when the foci are symmetric with respect to the circle $\Gamma$.
$\C(a)$ is a degenerate conic coinciding with the $y$-axis, which is outside  $\Gamma$ when the foci are symmetric with respect to that circle.
\end{remark}

We summarize the results of the above analysis as follows:

\begin{proposition}\label{prop:quad1}
If the focal points of the family $\C(t)$ either belong to $\Gamma$ or are symmetric with respect to $\Gamma$, then $\Gamma$ and $\C(t)$ are a $4$-Poncelet pair for each  $t\not\in\{a,b\}$.
\end{proposition}

\begin{proof}
The focal points belong to $\Gamma$ or they are symmetric with respect to $\Gamma$ if and only if $x_0=0$ and $a-b+y_0^2=1$ or, respectively, $y_0=0$ and $a-b-x_0^2+1=0$.
In both of those cases, we get that $\alpha_4=\beta_4=0$ (see Corollary \ref{cor:k4}), thus the corresponding Cayley condition $C_3(t)=0$ is satisfied for each $t$, which implies that $\Gamma$ and $\C(t)$ are a $4$-Poncelet pair for any $t$ for which $\C(t)$ is non-degenerate.
\end{proof}

Now we will provide two alternative proofs.

\noindent\emph{Second proof of Proposition \ref{prop:quad1}, based on synthetic geometry.} First we consider the case when the the foci $F_1$, $F_2$ are symmetric with respect to $\Gamma$.
In that case, the line $F_1F_2$ intersects $\Gamma$ at points $M$, $N$, such that $MN$ is a diameter of $\Gamma$, and:
\begin{equation}\label{eq:harmonic}
[F_1, F_2; M, N]=-1.
\end{equation}
Exactly one of the points $M$, $N$ is between $F_1$ and $F_2$, and we will assume that is $M$.

Let $\C$ be any conic with foci $F_1$, $F_2$ and $P$ be an intersection point of $\mathcal{C}$.
Since $\angle MPN=90^{\circ}$ and, from \eqref{eq:harmonic} we get $[PF_1,PF_2;PM,PN]=-1$, we get that the lines $PM$, $PN$ are the bisectors of the angles determined by the lines $PF_1$, $PF_2$.
Since $PF_1$, $PF_2$ satisfy the billiard reflection law off $\mathcal{C}$, we will have that one of the line $PM$, $PN$ is tangent to $\mathcal{C}$ and the other one is orthogonal to that conic at $P$.
Denote by $Q$ the point symmetric to $P$ with respect to the line $MN$.
Now it is easy to see that $\{P,Q\}=\mathcal{C}\cap\Gamma$.

If $\mathcal{C}$ is ellipse, then $N$ is outside that conic, and $M$ within it, so $PN$ and $QN$ are tangent to $\mathcal{C}$, thus $NPNQ$ is a polygonal line inscribed in $\Gamma$ and circumscribed about conic $\C$, i.e.~those two conics are a $4$-Poncelet pair.

On the other hand, if $\mathcal{C}$ is a hyperbola, $MPMQ$ is a polygonal line circumscribed about that conic and inscribed in $\Gamma$, so again those two conics are a $4$-Poncelet pair.

Second, we consider the case when both foci $F_1$, $F_2$ belong to the circle $\Gamma$.
Let $M$, $N$ denote the mid-points of two arcs of $\Gamma$ having the foci $F_1$, $F_2$ as the end-points.
Let $\C$ be any conic with foci $F_1$, $F_2$, such that $\Gamma$ is not entirely contained within that conic.

Denote by $P$ an intersection point of $\C$ and $\Gamma$.
Then lines $PM$ and $PN$ and the bisectors of the two angles determined by the lines $PF_1$ and $PF_2$, thus the focal property of the conics implies that one of $PM$, $PN$ is tangent to $\C$ and the other orthogonal to that conic.

Assume $PM$ is the tangent and let $Q$ be the point symmetric to $P$ with respect to $MN$.
We have that $MPMQ$ is inscribed in $\Gamma$ and circumscribed about $\C$, thus those two conics are a $4$-Poncelet pair.
\hfill$\Box$\medskip

\begin{corollary}\label{cor:twointersection}
	If the focal points of a conic $\C$  are symmetric with respect to $\Gamma$, then $\Gamma$ and $\C$ have at most two intersection points.
\end{corollary}

\begin{corollary}
	Let $F$ be a point on the circle $\Gamma$.
	Any circle $\mathcal{C}$ with the center at $F$ is $4$-Poncelet inscribed in $\Gamma$.
\end{corollary}
\begin{proof}
	Let $P$, $Q$ be the intersection points of the circles $\mathcal{C}$ and $\Gamma$ and let $N$ be the second intersection point of the line $FO$ with $\Gamma$.
	
	Then $NP$ is orthogonal to $FP$, since $FN$ is a diameter of $\Gamma$ and $P$ belongs to $\Gamma$.
	Thus, $NP$ is tangent to $\mathcal{C}$ at $P$.
	Similarly, $NQ$ is tangent to $\mathcal{C}$ at $Q$.
	Thus, $NPNQ$ is a Poncelet quadrilateral inscribed in $\Gamma$ and circumscribed about $\mathcal{C}$.
\end{proof}

\medskip
\noindent\emph{Third proof of Proposition \ref{prop:quad1}, based on the theory of invariants.}
Denote by $I_j$, $j\in\{1,2, 3, 4\}$, the coefficients of $D_t(\lambda)$ from \eqref{eq:Dt} with $\lambda^{4-j}$. The condition that there exists a quadrilateral inscribed in $\Gamma$ and circumscribed about $\mathcal{C}(t)$ can be expressed in terms of $I_j$:
\begin{equation}\label{eq:I4}
	8I_1I_4^2-4I_2I_3I_4+I_3^3=0.
\end{equation}
(see for example \cite{DR2011knjiga}*{Example 4.50}\footnote{There is a misprint in that example. Here we corrected the term $I_3^2$ to $I_3^3$}.

As said above, (i) the focal points belong to $\Gamma$ if and only if $x_0=0$ and $a-b+y_0^2=1$ and (ii)  they are symmetric with respect to $\Gamma$ if and only if $y_0=0$ and $a-b-x_0^2+1=0$.

In the case (i), the coefficients $I_j$ take the form (up to a common factor):
\begin{equation}\label{eq:Ii4}
	\begin{aligned}
		I_1&=-(a-t)(b-t);\\
		I_2&=I_1-t(b-a-1)+(b-a)a - b;\\
		I_3&=2(t-a);\\
		I_4&=-1.
	\end{aligned}
\end{equation}
By substituting the relations \eqref{eq:Ii4} for $I_j$ into \eqref{eq:I4} one gets the identity for all $t$.
Case (ii) can be treated analogously, which completes the proof.
\hfill$\Box$\linebreak

Finally, we will show that there are no other examples when $\Gamma$ and $\C(t)$ are a $4$-Poncelet pair for each $t$.

\begin{theorem}\label{th:isorotational4}
There are quadrangles inscribed in circle $\Gamma$ and circumscribed about conic $\C(t)$ for each $t\not\in\{a,b\}$
if and only if the focal points of the family $\C(t)$ either belong to $\Gamma$ or are symmetric with respect to $\Gamma$.
\end{theorem}

\begin{proof}
Suppose that $\mathcal{C}(t)$ is $4$-Poncelet inscribed in $\Gamma$ for each $t\not\in\{a,b\}$.
Thus the Cayley condition $C_3(t)=0$ is satisfied for infinitely many values of $t$.
Since $C_3(t)$ is a polynomial in $t$, we have that $C_3(t)=0$ everywhere.

By Corollary \ref{cor:k4}, $C_3(t)=\alpha_4+\beta_4t$ is a linear expression in $t$, so $\alpha_4=\beta_4=0$.
From there:
$$
\alpha_4+(a +b - x_0^2- y_0^2+3) \beta_4=0,
$$
which gives:
$$
1-(a-b)^2+(a-b-1) x_0^2+(b-a-1) y_0^2=0.
$$
From there, we get:
$$
y_0^2=\frac{(a-b-1) \left(a-b-x_0^2+1\right)}{-a+b-1},
$$
which we substitute into $\beta_4=0$:
$$
\frac{x_0^2 (a-b)^2 \left(a-b-x_0^2+1\right)}{(a-b+1)^2}=0.
$$
The first possibility, $x_0=0$, will imply $y_0^2=1+b-a$, which means that the foci belong to $\Gamma$.
The second case, $a-b-x_0^2-1=0$, implies $y_0=0$ and that the foci are symmetric with respect to $\Gamma$.
Finally, the case $a-b=0$ implies that $x_0^2+y_0^2=1$, i.e.~the origin is on $\Gamma$ and it is also the centre of the concentric family of circles $\mathcal{C}(t)$, which is a special case shown in Figure \ref{fig:quadrangles}, when $\C(t)$ are concentric circles.

The second half of the equivalence is Proposition \ref{prop:quad1}, thus the proof is concluded.
\end{proof}

{
\subsection{Hexagons}
\label{sec:k=6}

The condition for existence of a hexagon inscribed in $\Gamma$ and circumscribed about $\C(t)$ is:
$$
p_6(t)=\det\left(
\begin{array}{cc}
C_3(t) & C_4(t)\\
C_4(t) & C_5(t)
\end{array}
\right)
=0
$$
which is equivalent to:
\begin{equation}\label{eq:hex}
C_2(t)(\alpha_6+\beta_6t+\gamma_6t^2+\delta_6t^3)=0,
\end{equation}
with
$$
\begin{aligned}
&\delta_6=64 \left(
(a-b)^2 + 2 (a-b)(y_0^2- x_0^2)  + (x_0^2+y_0^2)^2
\right),
\\
&\gamma_6=128 (a - b) (a - b - x_0^2 + y_0^2)-\frac{\delta_6^2}{256}-\frac{\delta_6}4 (1 + 6 a + 6 b - 2 x_0^2 - 2 y_0^2).
\end{aligned}
$$
That condition is satisfied for each $t$ if and only if $\alpha_6=\beta_6=\gamma_6=\delta_6=0$.
Now, we see that $\gamma_6=0$ and $\delta_6=0$ imply $a=b$ or $a-b-x_0^2+y_0^2=0$.
Substituting $a=b$ to the expression for $\delta_6$, we get $x_0=y_0=0$, which means that $\C(t)$ are concentric circles, and $\Gamma$ shares the same centre with them.
Obviously, $\Gamma$ and $\C(t)$ cannot be a $6$-Poncelet pair for more than one value of $t$ in this case.

If $a-b=x_0^2-y_0^2$, the condition $\delta_6=0$ gives $x_0=0$ or $y_0=0$.
Since $a>b$, we can have only $y_0=0$ and the centre of $\Gamma$ is at one of the foci of the family $\C(t)$.
Now, calculating the coefficients $\alpha_6$, $\beta_6$ and substituting $y_0=0$, $x_0^2=a-b$, we get:
$$
\begin{aligned}
&\alpha_6=-(4 a-4 b-1) \left(16 a^2-16 a b+8 a-4 b-3\right),
\\
&\beta_6=4 (4 a-4 b-1) (4 a-4 b+1).
\end{aligned}
$$
Since $a>b$, the condition $\alpha_6=\beta_6=0$ implies $4a-4b=1$.
Thus, the distance between the foci equals $2\sqrt{a-b}=1$, so we can conclude:
\begin{theorem}\label{th:isorotational6}
Each non-degenerate conic $\mathcal{C}(t)$ is $6$-Poncelet inscribed in $\Gamma$ if and only if one of the focal points belongs to $\Gamma$ and the other one is at the centre of $\Gamma$, see Figure \ref{fig:hex}.
\end{theorem}
}
\begin{figure}[h]
	\centering
	\begin{tikzpicture}[scale=4]
		\draw[thick](0.5,0) circle (1);
		\draw[black,fill=gray](0.5,0) circle (0.02);
		\draw[black,fill=gray](-0.5,0) circle (0.02);
		\draw[thick,gray](0,0) ellipse (0.774597 and 0.591608);
		\draw[thick,gray](0,0) ellipse (1 and 0.866025);
		
		\draw (0.5, 1.)--(1.04101, -0.841014)--(-0.341014, -0.541014)--(0.5, -1.)--(1.04101, 0.841014)--(-0.341014, 0.541014)--(0.5, 1.);
		\draw[black,fill=black](0.5, 1.) circle (0.02);
		\draw[black,fill=black](1.04101, -0.841014) circle (0.02);
		\draw[black,fill=black](-0.341014, -0.541014) circle (0.02);
			\draw[black,fill=black](0.5, -1.) circle (0.02);
		\draw[black,fill=black](1.04101, 0.841014) circle (0.02);
		\draw[black,fill=black](-0.341014, 0.541014) circle (0.02);
		\draw[gray,fill=gray!50](0.755794, 0.12956) circle (0.02);
		\draw[gray,fill=gray!50](-0.211765, -0.56907) circle (0.02);
		\draw[gray,fill=gray!50](-0.450339, -0.481349) circle (0.02);
		\draw[gray,fill=gray!50](0.755794, -0.12956) circle (0.02);
		\draw[gray,fill=gray!50](-0.211765, 0.56907) circle (0.02);
		\draw[gray,fill=gray!50](-0.450339, 0.481349) circle (0.02);
		
		\draw(1.37758, 0.479426)--(0.553871, -0.998548)--(0.0685467, -0.902135)--(1.37758, -0.479426)--(0.553871, 0.998548)--(0.0685467, 0.902135)--(1.37758, 0.479426);
			\draw[black,fill=black](1.37758, 0.479426) circle (0.02);
		\draw[black,fill=black](0.553871, -0.998548) circle (0.02);
		\draw[black,fill=black](0.0685467, -0.902135) circle (0.02);
		\draw[black,fill=black](1.37758, -0.479426) circle (0.02);
		\draw[black,fill=black](0.553871, 0.998548) circle (0.02);
		\draw[black,fill=black](0.0685467, 0.902135) circle (0.02);
		\draw[gray,fill=gray!50](0.900587, -0.37644) circle (0.02);
		\draw[gray,fill=gray!50](-0.223582, -0.844102) circle (0.02);
		\draw[gray,fill=gray!50](0.349375, -0.811451) circle (0.02);
		\draw[gray,fill=gray!50](0.900587, 0.37644) circle (0.02);
		\draw[gray,fill=gray!50](-0.223582, 0.844102) circle (0.02);
		\draw[gray,fill=gray!50](0.349375, 0.811451) circle (0.02);
		
	\end{tikzpicture}
	\caption{Each conic from the confocal family is inscribed in hexagons whose vertices belong to the circle. The foci are symmetric with respect to the circle.}\label{fig:hex}
\end{figure}

\noindent
\emph{Alternative geometric proof of one direction of Theorem \ref{th:isorotational6}.}

Denote by $O$ and $r$ the center and the radius of $\Gamma$.
Let $F_1$ be a point of $\Gamma$; $\C$ an ellipse with foci $F_1$ and $O$; $A$ an intersection point of $\Gamma$ and $\C$;
$B$ the point of $\Gamma$ such that $AB$ is tangent to $\C$;
 $E$, $D$ the points symmetric to $A$, $B$ respectively with respect to $OF_1$;

Our goal is to show that closed polygonal line $ABDEDBA$ is circumscribed about $\mathcal C$.

Denote by $C$ the intersection of $BD$ and the line $F_1O$. We want to show that $C$ and $A$ belong to the same ellipse with foci $F_1$ and $O$. Denote $F_1A=:\ell$. We observe that $F_1A+OA=\ell+r$. Similarly, $F_1C+OC=r+2OC$. Thus, we need to show that $2OC=\ell$.

 Denote by $G$ any point on the line $AB$, which is at the opposite to $B$ side of $A$. Observe that $\angle OAB=\angle F_1AG$ and $\angle OAB=\angle OBA$. The first fact follows from the focal property of ellipses and the second from $\triangle OAB$ being isosceles, since $O$ is the center of the circle $\Gamma$. From there we get $F_1A||OB$.

Let $OO_1$ be the altitude of the isosceles $OF_1A$. Construct  the normal line from $O_1$ to $F_1O$, with $O_2$ denoting  the footing of the normal at $F_1O$. We have that
$\angle F_1O_1O_2=\angle OBC$ as angles with parallel sides. Similarly, $\angle F_1O_1O_2=\angle F_1OO_1$ as angles with orthogonal sides. Thus, $\angle F_1OO_1=\angle OBC$.

We showed that the triangles $\triangle F_1OO_1$ and $\triangle OBC$ are congruent. Thus, $F_1O_1=OC $ and we finally get:
$$
OC=\frac{1}{2}F_1A.
$$

Similar proof is valid when the conic is a hyperbola.
\hfill$\Box$\medskip

\begin{remark}
In the case when $\C$ is ellipse, we assume in the proof above that $\Gamma$ is not entirely contained within the ellipse $\C$.
Thus, the proof implies that the circle
and the ellipse have at most two intersection points.
\end{remark}

\section{Do isoperiodic families  exist for $n\ne4,6$?}\label{sec:noisorot}

Any Poncelet $n$-polygon can be also considered as a closed $mn$-polygon, for any $m\in\mathbb{N}$, where its edges and vertices are repeated $m$ times each.
This is our motivation to refine the notion of $n$-Poncelet pairs from Definition \ref{def:n-pair}.

\begin{definition}\label{def:n-pair-strict}
We will say that conics $\Gamma$ and $\C$ are a \emph{strict $n$-Poncelet pair} or that $\C$ is \emph{strictly $n$-inscribed in $\Gamma$} if the following conditions are satisfied:
\begin{itemize}
	\item there is a polygon with $n$ sides inscribed in  $\Gamma$ and circumscribed about $\C$;
	\item for any $m$ that divides $n$ and distinct from $n$, there is no polygonal line with $m$ sides inscribed in $\Gamma$ and circumscribed about $\C$.
\end{itemize}
\end{definition}

Notice that Cayley's conditions also reflect the existence of strict and non-strict Poncelet pairs of conics, since, for non-prime numbers, thos conditions have factors corresponding to the divisors of the numbers, see \cite{Cayley1861}.
For example, in the calculations of this paper, we can see that the condition for hexagons \eqref{eq:hex} has a factor corresponding to triangles, i.e.~$p_6(t)=p_3(t)q_6(t)$, for some polynomial $q_6(t)$.

More generally, for $m|n$, we have that $p_{m}(t)=0$ implies $p_{n}(t)=0$.
We will denote by $q_n(t)$ the largest factor of $p_n(t)$ which is co-prime with all polynomials $p_m(t)$, where $m|n$ and $m<n$.

\begin{proposition}\label{prop:strict-pair}
For $t\notin\{a, b\}$, the conics
$\Gamma$ and $\C(t)$ form a strict $n$-Poncelet pair if and only if $q_n(t)=0$.
\end{proposition}

\begin{lemma}\label{lema:a,b}
Suppose that $\Gamma$ and $\C(t)$ form an $n$-Poncelet pair for infinitely many values of $t$.
If $m$ is co-prime with $n$, then all values of $t$ that satisfy the Cayley's condition for $m$ belong to the set $\{a,b\}$.
\end{lemma}
\begin{proof}
Notice first that, in the condition $p_n(t)=0$ for the $n$-Poncelet pair, $p_n(t)$ is a polynomial in $t$, where the coefficients are polynomials in $a$, $b$, $x_0^2$, $y_0^2$.
If the polynomial $p_n(t)$ has infinitely many zeroes, then it must be identically equal to zero.
This means that $\Gamma$ and $\C(t)$ form an $n$-Poncelet pair in the complexified plane for each $t\in\mathbb{C}\setminus\{a,b\}$.

From there, if $m$ and $n$ are co-prime, the polynomial $p_{m}(t)$ cannot have any zeroes in $\mathbb{C}$, other than $a$ and $b$.
\end{proof}

\begin{proposition}\label{prop:singular}
Suppose that $\Gamma$ and $\C(t)$ form a strict $n$-Poncelet pair for infinitely many values of $t$.
If $m$ is not a multiple of $n$, then all values of $t$ that satisfy the Cayley's condition $p_m(t)=0$ must be in the set $\{a,b\}$.
\end{proposition}
\begin{proof}
Similarly as in the proof of Lemma \ref{lema:a,b}, we can show that the polynomial $q_n(t)$ is identically equal to zero.
This implies that $\Gamma$ and $\C(t)$ form a strict $n$-Poncelet pair in the complexified plane for each $t\in\mathbb{C}\setminus \{a, b\}$.

If for some $t\in\mathbb{C}\setminus \{a, b\}$, there exists $m\ne n$ such that $p_{m}(t)=0$, then either $m|n$ or $n|m$. The former is impossible due to the definition of a strict $n$-Poncelet pair and the latter by the assumption of the Proposition, which concludes the proof.
\end{proof}

\begin{lemma}\label{lema:t3=a}
If $t_3=a$, then $\Gamma$ and $\C(t)$ form a $4$-Poncelet pair for each $t$.
\end{lemma}
\begin{proof}
We have that $t_3=a$ is equivalent to:
$$
(x_0^2+y_0^2-a+b-1)^2+4(a-b)y_0^2=0.
$$	
Since $a\ge b$, that implies $(a-b)y_0=0$ and $x_0^2+y_0^2-a+b-1=0$.
If $a=b$, we get $x_0^2+y_0^2=1$, i.e.~$\C(t)$ is a family of concentric circles with the centre on $\Gamma$.
If $y_0=0$, we have $x_0^2=a-b+1$.

According to the discussion from Section \ref{sec:k=4}, each of these conclusions implies that $\Gamma$ and $\C(t)$ form a $4$-Poncelet pair for each $t$.
\end{proof}

\begin{theorem}\label{th:noisorot}
If $\Gamma$ and $\mathcal{C}(t)$ form $n$-Poncelet pairs for infinitely many values of $t$, then  $n\in\{4,6\}$.
\end{theorem}
\begin{proof}
Suppose that $n$ is a number not equal to $4$ or $6$, such that $\Gamma$ and $\mathcal{C}(t)$ form $n$-Poncelet pairs for infinitely many values of $t$.
Based on Corollary \ref{cor:k3}, there is a unique value $t=t_3$, which if given by \eqref{eq:t3}, such that the Cayley condition for triangles inscribed in $\Gamma$ and circumscribed about $\mathcal{C}(t)$ is satisfied.
Thus $n\neq3$.

Proposition \ref{prop:singular} and Lemma \ref{lema:t3=a} then imply $t_3=b$.

According to Corollary \ref{cor:k4}, there is a quadrangle inscribed in $\Gamma$ and circumscribed about $\mathcal{C}(t)$ if and only if $\alpha_4+\beta_4t=0$. After substituting $t_3=b$, this is equivalent to:
$$
(b-t) \left(1-a+b-x_0^2-y_0^2\right)=0.
$$
Since $n\neq4$, we have:
\begin{equation}\label{eq:inequality}
1-a+b-x_0^2-y_0^2\neq0.
\end{equation}
Thus, the equation $\alpha_4+\beta_4t=0$ has a unique solution $t=b$.

{ Let us consider now Poncelet pentagons.}
Since the coefficients $C_2(t)$ and $C_3(t)$ both take values $0$ for $t=b$, that means that $t=b$ is also a solution of the cubic equation from Corollary \ref{cor:k5}, { describing Poncelet pentagon}, i.e.
$$
\alpha_5+\beta_5b+\gamma_5b^2+\delta_5b^3=0.
$$
Substituting $t_3=b$ in the expressions from Corollary \ref{cor:k5} { for $\alpha_5$, $\beta_5$, $\gamma_5$, and $\delta_5$}, we get:
\begin{equation}\label{eq:poly5}
\alpha_5+\beta_5t+\gamma_5t^2+\delta_5t^3
=
-64 (b-t)^2\times
\left((t-a) \left(2 a-2 b+2 x_0^2+2 y_0^2-1\right)+a-b\right).
\end{equation}
Thus $n=5$ only if:
$$
2 a-2 b+2 x_0^2+2 y_0^2-1=a-b=0,
$$
which implies $2x_0^2+2y_0^2=1$.
Under those conditions, we calculate that $t_3=b-\frac1{16}$, which is in contradiction with the assumption $t_3=b$.
We conclude that $n\neq5$.

The equality \eqref{eq:poly5} shows that
the polynomial $\alpha_5+\beta_5t+\gamma_5t^2+\delta_5t^3$ has a double root at $t=b$.
{
Considering} { the third root $t_5$ of that polynomial, { if it exists,} we have the following:
\begin{itemize}
	\item either $t_5\in\{a,b\}$;
	 or
	\item the polynomial $\alpha_5+\beta_5t+\gamma_5t^2+\delta_5t^3$ does not have the third root, i.e.~$\delta_5=0$.
\end{itemize}
}
We split those possibilities into three subcases.

\emph{Case 1: The third root of  $\alpha_5+\beta_5t+\gamma_5t^2+\delta_5t^3$ equals $a$.}
 From \eqref{eq:poly5}, we see that then $a=b$.
Substituting that into $t_3=b$, we get $x_0^2+y_0^2=1$, which implies $n=4$.

\emph{Case 2: The third root of  $\alpha_5+\beta_5t+\gamma_5t^2+\delta_5t^3$ equals $b$.}
Relation \eqref{eq:poly5} then implies $a-b+x_0^2+y_0^2-1=0$, which is in contradiction with \eqref{eq:inequality}.

\emph{Case 3: The polynomial $\alpha_5+\beta_5t+\gamma_5t^2+\delta_5t^3$ does not have the third root, i.e.~$\delta_5=0$.}
That means:
$2 a-2 b+2 x_0^2+2 y_0^2-1=0$.
Combining that with $t_3=b$, we get
$16(a - b) x_0^2=1$ and $16(a-b)y_0^2=-(4 a-4 b-1)^2$.
Since $a\ge b$, those two equations give $4a-4b-1=0$, $y_0=0$, and $x_0=\pm\frac12$.
According to the discussion from Section \ref{sec:k=6}, we have $n=6$.

This concludes the proof.
\end{proof}

\section{Poncelet pairs of conics and Blaschke ellipses}\label{sec:blaschke}

\subsection{Properties of the Blaschke transformations}

In this section, we are going to review basic properties of Blaschke transformations in a way we will need it later.
For a more comprehensive overview, see e.g.~\cite{DGSV} and references therein.
Let us denote
\begin{equation}\label{eq:Blaschke}
	B_p(z)=\frac{p-z}{1-z\bar p}
\end{equation}
a M\"obius transformation for a given $p\in \mathbb{C}$.

\begin{lemma} Let $z, w, v\in  \mathbb{C}$, $z\ne \bar w$, $z\bar v\ne 1$, $v\bar w\ne 1$. Then
	$$
	\frac{1-B_w(z)\bar {B_w(v)}}{1-z\bar v}=\frac{1-|w|^2}{(1-z\bar w)(1-\bar v w)}.
	$$
\end{lemma}
\begin{proof} It follows by straightforward calculation:
	$$
	1-B_w(z)\bar {B_w(v)}=\frac{(1-z\bar w)(1-\bar v w)-(z-w)(\bar v-\bar w)}{(1-z\bar w)(1-\bar v w)}=\frac{(1-z\bar w)(1-|\bar w|^2)}{(1-z\bar w)(1-\bar v w)}.
	$$
\end{proof}
\begin{corollary} For $v=z$ it holds:
	\begin{equation}
		\frac{1-|B_w(z)|^2}{1-|z|^2}=\frac{1-|w|^2}{|1-z\bar w|^2}.
	\end{equation}
\end{corollary}

\begin{corollary}\label{cor:Mobius2} The M\"oblius transformations \eqref{eq:Blaschke} have the following properties:
	\begin{itemize}
		\item If $|z|<1$ and $|w|<1$ then $|B_w(z)|<1$.
		\item If $|z|<1$ and $|w|\ge 1$ then $|B_w(z)|\ge 1$.
		\item If $|z|>1$ and $|w|<1$ then $|B_w(z)|>1$.
		\item If $|z|>1$ and $|w|>1$ then $|B_w(z)|<1$.

	\end{itemize}
\end{corollary}

\begin{lemma} \label{lemma:Blaschke2}
	For given $p, q$, $|p|<1, |q|<1$ there exists a unique $c$ such that $B_c(p)=q$. For such a $c$, it is $|c|<1$.
\end{lemma}
\begin{proof}
	For given $p, q$ the condition $B_c(p)=q$ leads to a system of two linear equations on two real variables $c_1=\Re c, c_2=\Im c$. The determinant of that system
	is
	$$
	\Delta=1-|p|^2|q|^2.
	$$
	Thus, from  $|p|<1, |q|<1$ it follows that $\Delta \ne 0$, and the system has a unique solution, $c$. From Corollary \ref{cor:Mobius2} it follows that $c<1$.
\end{proof}

Let us denote $\mathbb T=\{z||z|=1\}$ the unit circle and $\mathbb D=\{z||z|<1\}$ the open unit disk bounded by $\mathbb T$.

The transformations defined by \eqref{eq:Blaschke} are called the Blaschke transformations. For $|p|<1$, $B_p$ is an automorphism of $\mathbb D$, see \cite{DGSV}. Composition of $n$ Blaschke transformations defines a Blaschke product $B$ of degree $n$. Following \cite{DGSV}, one associates to $B$ a Blaschke curve as the caustic of all polygons with vertices at the solutions of $B(z)=\lambda$, for each $\lambda \in \mathbb T$.
Let us recall two statements from \cite{DGSV}.

\begin{proposition}[Theorem 2.9, \cite{DGSV}]\label{prop:th2.9}
	Let $B$ be a Blaschke product of degree $3$ with zeros $0, p, q$. For $\lambda \in \mathbb T$, let $z_1, z_2, z_3$ denote the three distinct solutions to $B(z)=\lambda$. Then the lines joining $z_j$ and $z_k$, ($j\ne k$) are tangent to the ellipse given by
	$$
	|w-p|+|w-q|=|1-\bar p q|.
	$$
\end{proposition}

\begin{proposition} [Corollary 13.9, \cite{DGSV}]
	Let $B$ be a Blaschke product of degree $4$. Then the associated Blaschke curve is an ellipse if and only if $B$ is decomposable.
\end{proposition}

\subsection{Poncelet triangles and Blaschke ellipses}

Now, we are going to study triangles inscribed in the unit circle and having a conic from a given confocal family as the caustic. Unlike \cite{DGSV} we, in principle, allow the possibility for the inscribed conic to be a hyperbola.

{
	
	\begin{theorem}\label{th:Bla3}
		Given two points $p, q\in \mathbb D$. Then there exists a unique conic $\C$ with the foci $p, q$, which is $3$-Poncelet inscribed in $\mathbb T$. Moreover, $\C$ is an ellipse. That ellipse is the Blaschke ellipse with the major axis of length $|1-\bar p q|$.
	\end{theorem}
	
	\begin{proof}
		We can construct the Blaschke $3$-ellipse with the focal points $p$, $q$.
		This ellipse satisfies all the requirements of the Theorem. It is unique such a conic, because the Cayley-type condition derived above in Corollary \ref{cor:k3} shows that there is at most one conic with the foci $p$, $q$ which is a $3$-Poncelet caustic with respect to $\mathbb T$.
	\end{proof}
	
}

{
	\begin{remark} Previously existing proofs of the uniqueness, like one from \cite{DGSV} (pp. 43– 44), assume that the conics are completely situated within $\mathbb D$. They are based on a nice and intuitive argument, which however, does not seem to be applicable nor easily extendable in a more general situation of our interest, when conics are not necessarily completely situated within $\mathbb D$.
		Moreover, in principle, a triangle which is inscribed in $\mathbb T$ and whose sides belong to tangents to an ellipse which is not completely contained in $\mathbb D$ could exist, see Figure \ref{fig:triangle}. Our nontrivial conclusion is that in such a situation, at least one focal point lies outside $\mathbb D$.
	\end{remark}
}
\begin{figure}[h]
	\begin{center}
		\begin{tikzpicture}[scale=4]
			\draw[thick](1,1) circle (1);
			\draw[black,fill=gray](0.4472,0) circle (0.02);
			\draw[black,fill=gray](-0.4472,0) circle (0.02);
			\draw[thick,gray](0,0) ellipse (0.6 and 0.4);
			
			\draw (1., 2.)--(0.619984, 0.0750201)--(0.0141624, 0.832297)--(1., 2.);
			\draw[black,fill=black](1,2) circle (0.02);
			\draw[black,fill=black](0.619984, 0.0750201) circle (0.02);
			\draw[black,fill=black](0.0141624, 0.832297) circle (0.02);
			
			\draw[gray,fill=gray!50](0.59487, -0.0521934) circle (0.02);
			\draw[gray,fill=gray!50](-0.52287, 0.196193) circle (0.02);
			\draw[gray,fill=gray!50](0.529412, 0.188235) circle (0.02);

		\end{tikzpicture}
		\caption{A triangle inscribed in the circle}\label{fig:triangle}
	\end{center}
\end{figure}

\begin{remark} Not only that our proof is more { elaborated and with a wider scope} than previously existing ones, but in our statement { we sharpened the focus} as well. We prove that $\C$ is unique not only among the ellipses with focal points $p$, $q$ but it is also unique among all conics with these focal points. In other words, we prove that there is no hyperbola with given properties.
\end{remark}
\begin{theorem}
	Given a triangle $ABS$ inscribed in $\mathbb T$ with the property that lines $AB$, $AC$, $BC$ touch an ellipse in the points $C_1$, $B_1$, $A_1$ respectively with the following orders of the points $B-A_1-C$, $C_1-B-A$, and $B_1-C-A$. Then at least one focal point of $\C$ is outside $\mathbb D$.
\end{theorem}
\begin{remark} Let us observe Figure  \ref{fig:triangle} and see what happens with rotational numbers in such a situation. Let us denote by $\rho_A$, $\rho_B$, $\rho_C$ the arcs of the ellipse $\C$ seen from the corresponding vertex of the triangle $ABC$. Then it is obviously
	$$
	\rho_A=\rho_B+\rho_C.
	$$
	
\end{remark}

\subsection{Poncelet quadrangles and Blaschke ellipses}

\begin{remark}
	We want to observe first that  not all $4$-Poncelet ellipses are Blaschke ellipses.
\end{remark}
However, the following reformulation and generalization of a statement from \cite{DGSV}, p. 175, holds:

	\begin{theorem}\label{th:Bla4}
		Given $p, q\in \mathbb D$. Among all conics with the foci $p$, $q$ there is a unique one which is $4$-Poncelet inscribed in $\mathbb T$. It is an ellipse and it is the Blaschke ellipse.
		
	\end{theorem}
	\begin{proof} Using Lemma \ref{lemma:Blaschke2}, we see that for given $p, q\in \mathbb D$, there is a unique $c$ (this $c\in \mathbb D$), such that $B_c(p)=q$. Using now the construction from   \cite{DGSV}
		one gets the Blaschke product with zeros $0$, $c$, $p$, $q$. Since $B_c(p)=q$, $B_c(0)=c$, this product is decomposable and thus, generates a Blaschke ellipse which is $4$-Poncelet inscribed in $\mathbb{T}$.
		Using the Cayley-type condition derived above in Corollary \ref{cor:k4} we see that there is at most one such $4$-Poncelet inscribed conic with the foci $p$, $q$.
		Thus, we obtained the statement.
	\end{proof}

The last Theorem has a highly counter-intuitive consequence.
\begin{corollary}\label{cor:no4hyp} A hyperbola with the foci in $\mathbb D$ cannot be a $4$-Poncelet inscribed in $\mathbb T$.
\end{corollary}

\section{View towards Painlev\'e VI equations}\label{sec:painleve}

\subsection{Picard solutions and Okamoto transformation}
\label{sect_Picard}

The Painlev\'e VI equations form one of the six families of remarkable second order ordinary differential equations introduced by Paul Painlev\'e and his school at the beginning of the XX century.
The Painlev\'e VI equations are the four-parameter family
\begin{equation}
	\label{Painleve}
	\frac{d^2 y}{dx^2} = \frac{1}{2} \left(  \frac{1}{y} + \frac{1}{y-1}+ \frac{1}{y-x}  \right) \left( \frac{dy}{dx} \right)^2 - \left( \frac{1}{x} + \frac{1}{x-1} + \frac{1}{y-x} \right) \frac{dy}{dx}
\end{equation}
\begin{equation*}
	+ \frac{y(y-1)(y-x)}{x^2(x-1)^2}\left( \alpha +\beta\frac{x}{y^2} + \gamma\frac{x-1}{(y-1)^2} + \delta\frac{x(x-1)}{(y-x)^2}  \right)
\end{equation*}
with parameters $\alpha, \beta, \gamma, \delta \in \mathbb C$, see for example \cites{Okamoto1987, IKSYGaussPainleve, Hitchin,  DS2019} and references therein. We are here interested in two particular equations with parameter values
\begin{equation}
	\label{constants}
	\alpha=\frac{1}{8}, \qquad \beta=-\frac{1}{8}, \qquad \gamma=\frac{1}{8}, \qquad \delta=\frac{3}{8},
\end{equation}
and
\begin{equation}
	\label{constants1}
	\alpha=0, \qquad \beta=0, \qquad \gamma=0, \qquad \delta=\frac{1}{2}.
\end{equation}

We will use  the transformed Weierstrass $\wp$-function with periods $2w_1$ and $2w_2$ satisfying the equation
\begin{equation}
	\label{wp}
	\left( {\hat \wp^\prime}(z) \right)^2 = \hat \wp(z)(\hat \wp(z)-1)(\hat \wp(z)-x).
\end{equation}
Apparently, more than a decade prior to Painlev\'e, Picard found the explicit solution of the above mentioned Painlev\'e VI equation with parameters \eqref{constants1}, see \cite{Picard}. Namely,  the function
\begin{equation}
	\label{Picard}
	y_0 (x) = \hat \wp(2c_1 w_1(x) + 2c_2 w_2(x))
\end{equation}
is the Picard solution \cite{Picard} of the Painlev\'e VI equation (\ref{Painleve}) with  constants $\alpha = \beta=\gamma=0,$ $\delta=1/2$, see for example (3.6) in \cite{Okamoto1987}.

The   transformation
\begin{equation}
	\label{Okamoto}
	y(x) = y_0 + \frac{y_0(y_0-1)(y_0-x)}{x(x-1)y_0^\prime - y_0(y_0-1)},
\end{equation}
follows from \cite{Okamoto1987} and was written in this form in \cite{DS2019}. It relates the Picard solution
$y_0(x)$ (\ref{Picard}) and the general solution $y(x)$ of the
Painlev\'e VI equation with constants (\ref{constants}).

In order to apply the Okamoto transformation, one needs an expression for $y_0'.$ Such an expression is obtained in \cite{DS2019} as follows.
Consider the elliptic curve $\E$ parameterized by $(u, v)=(\hat \wp, \hat \wp')$, that is the curve of the equation
\begin{equation}
	\label{elliptic}
	v^2=u(u-1)(u-x)
\end{equation}

and its Abel map
\begin{equation*}
	\mathcal A(p )=\int_{p_\infty}^p\frac{du}{v},  \qquad p\in\mathcal L,
\end{equation*}
into the Jacobian $J(\E) = {\mathbb C}/\{n(2w_1)+m (2w_2)\mid n, m\in{\mathbb Z}\} $,
where $p_\infty$ is the point at infinity on the compactified curve $\E$. Let us now take an arbitrary point $2w_1c_1+2w_2c_2$ in the Jacobian, with some constants $c_1$ and $c_2$ and consider its Jacobi inversion, that is the point $q_0$ on the curve such that

\begin{equation*}
	\mathcal A(q_0 )=\int_{p_\infty}^{q_0}\frac{du}{v} = 2w_1c_1+2w_2c_2.
\end{equation*}
From these data, denoting by $q^*$ the elliptic involution of a point $q$ of the curve, we build the following differential of the third kind:
\begin{equation}\label{eq:omega1ds}
	\Omega_1:=\Omega_{q_0,q_0^*} - 4 \pi{\rm i}c_2\omega\,,
\end{equation}
with $\Omega_{q_0,q_0^*}$ being the normalized differential of the third kind with poles at $q_0$ and $q_0^*$ of residues $1$ and $-1$, respectively, and $\omega$ the normalized holomorphic differential $w=\frac{1}{2w_1}\frac{du}{v}\,.$  Differential $\Omega_1$ has two zeros paired by the elliptic involution, we denote by $y$ their projection on the $u$-sphere. This projection, as a function of $x$, satisfies a Painlev\'e-VI equation, as stated in the following theorem.
\begin{theorem}[\cite{DS2019}, Theorem 1]\label{th:ds2019t1} If $x$ varies, while $c_1, c_2$ stay fixed, then
	$y(x),$ the position of zero of $\Omega_1$, is the Okamoto transformation \eqref{Okamoto} of $y_0$ and satisfies Painlev\'e VI with constants \eqref{constants}
\end{theorem}
Moreover, the derivative of $y_0'$ needed to apply the Okamoto transformation \eqref{Okamoto} can also be expressed in terms of $\Omega_1:$
\begin{equation}
	\label{derivative}
	\frac{dy_0}{dx} = - \frac{1}{4} \Omega_1(p_x) \frac{\omega(p_x)}{\omega{(q_0)}}.
\end{equation}
Note that the $a$-period of $\Omega_1$ is $-4\pi i c_2$, while the $b$-period is $4\pi i c_1$. Thus, the periods are preserved under the deformation described in Theorem \ref{th:ds2019t1}.

\subsection{From isoperiodic families to explicit solutions to a Painlev\'e VI equation}

	\subsubsection{Quadrangles}

In order to simplify calculations, we will consider the ``super-symmetric"  case, when both $x_0$ and $y_0$ are zero, meaning that the circle $\Gamma=\Gamma_0$ is concentric with the confocal conics $\mathcal{C}(t)$. Then, $a=b+1$ and
$$
\Gamma_0:\ x^2+y^2=1,
$$
and
$$
\mathcal{C}(t):\ \frac{x^2}{a-t}+\frac{y^2}{a-t-1}=1.
$$
The pencil of cubics $\E_t$ given by \eqref{eq:elliptic-curves} under the assumptions we just made becomes:
$$
\begin{aligned}
	\E_t: \mu^2=
	-(\lambda +1) (\lambda(a-t)+1) (\lambda(a-t-1)+1).
\end{aligned}
$$
This pencil has a section of points $P_0(t)$ of order $4$ corresponding to $\lambda=0$ each.
Each of the cubics $\E_t$ has three finite ramification points:
$$
\lambda_1(t)=\frac{1}{t-a},\quad \lambda_2(t)=\frac{1}{t-a+1},\quad \lambda_3(t)=-1.
$$
Let us consider the following M\"obius transformations of the basis:
$$
u=\phi_t(\lambda)=\frac{\lambda-\lambda_1}{\lambda_2-\lambda_1}.
$$
Then $\phi_t(\infty)=\infty$, $\phi_t(\lambda_1)=0$, $\phi_t(\lambda_2)=1$ and denote $x:=\phi_t(\lambda_3)$.
The M\"obius transformations $\phi_t$ induce morphisms from the elliptic curves $\E_t$ onto the elliptic curve \eqref{elliptic}.
Since $P_0(t)$ corresponding to $\lambda=0$ are points of a fixed finite order (that order equals $4$), we get the following:
\begin{proposition} Let
	\begin{equation}\label{eq:y_0}
		y_0(t):=\phi_t(0).
	\end{equation}
	Then $y_0(t)=-(a-t-1)$ is a Picard solution to the Painlev\'e VI $(0,0,0, 1/2)$ as a function of $x=(a-t-1)^2$.
\end{proposition}
\begin{proof} By a direct computation we get
	\begin{equation}\label{eq:ty_0}
		x=\phi_t(\lambda_3)=(a-t-1)^2, \quad y_0(t)=\phi_t(0)=-(a-t-1).
	\end{equation}
	
	Since the points $P_0(t)$ which correspond to $\lambda=0$ have all order $4$, we see that for them $c_1$, $c_2$ do not depend on $t$, thus also do not depend on $x$, so, by \eqref{Picard}, they provide the Picard solution to the Painlev\'e VI $(0,0,0, 1/2)$.
\end{proof}

Before we apply the Okamoto transformation \eqref{Okamoto}, let us verify one technical statement.
\begin{lemma}\label{lemma:identity}  For $y_0$ from \eqref{eq:y_0} the identity holds:
	\begin{equation}\label{eq:identity}
		\frac{(y_0-1)(y_0-x)}{x(x-1)y_0^\prime - y_0(y_0-1)}=-2.
	\end{equation}
\end{lemma}
\begin{proof} Observe that
	\begin{equation}\label{eq:diffy_0}
		\frac{dy_0}{dx} = \frac{1}{2y_0}.
	\end{equation}
	The rest is a straightforward verification.
\end{proof}
\begin{theorem}\label{th:painleve}  Let
	\begin{equation}\label{eq:y}
		y(t):=-y_0(t).
	\end{equation}
	Then $y(t)=a-t-1$ is a solution to the Painlev\'e VI $(1/8,-1/8,1/8, 3/8)$ as a function of $x=(a-t-1)^2$.
\end{theorem}
\begin{proof} By applying the Okamoto transformation \eqref{Okamoto} and using Lemma \ref{lemma:identity} we get:
	$
	y=y_0-2y_0=-y_0.
	$
\end{proof}

\begin{remark}
	Both solutions $y=y(x)$ to Painlev\'e VI $(1/8,-1/8,1/8, 3/8)$ and $y_0=y_0(x)$ to Painlev\'e VI $(0,0,0, 1/2)$ can be seen as instances of the solution parametrized as $y(s)=s$, $x(s)=s^2$ where $s=a-t-1$ in one case and $s=-(a-t-1)$ in the other case.
	So, in some sense, here the Okamoto transformation maps the ``same" solutions to each other, although the Okamoto transformation is not identity itself. It is an involution in this case. Let as also observe that the same solution was obtained by Hitchin in \cite{Hitchin} in his consideration of points of order four on elliptic curves. Hitchin also observed in the same paper that $y(s)=s$, $x(s)=s^2$ provides solutions to all Painlev\'e VI with
	$$
	\alpha + \beta=0,\quad \gamma +\delta=\frac{1}{2},
	$$
	in \eqref{Painleve}. However, Hitchin did not consider the Okamoto transformation nor iso-periodic confocal pencils.
\end{remark}
\begin{remark}
	We did not need to use \eqref{derivative} here since our iso-rotational confocal family provided an explicit form of the one-parameter family of cubics $\E_t=\E_t(\mu, \lambda)$ with the section of marked points of order four lying above $\lambda=0$. This section provides an explicit form of $y_0$ (see \eqref{eq:ty_0}), which can be differentiated directly, as in \eqref{eq:diffy_0}.
\end{remark}

	\subsubsection{Hexagons}
	
	Here $x_0=1/2$ and $y_0=0$ and $b=a-1/4$.
	Thus,
	$$
	\Gamma_0:\ \left(x-\frac{1}{2}\right)^2+y^2=1,
	$$
	and
	$$
	\mathcal{C}(t):\ \frac{x^2}{a-t}+\frac{y^2}{a-t-\frac{1}{4}}=1.
	$$
The pencil of cubics \eqref{eq:elliptic-curves} is then given by:
	$$
		\E_t: \mu^2=
		-\frac{1}{16} ((4s-1) \lambda +4)
		\left(4s \lambda^2 +(4  s+3) \lambda+4\right)\\
	$$
	where $s=a-t$.
	
	Each of the cubics $\E_t$ has three finite ramification points:
	$$
	\lambda_1(t)=-\frac{4}{4s-1},\quad \lambda_{2,3}(t)=-\frac{1}{2}-\frac{3}{8s}\pm\frac{\sqrt{16s^2-40s+9}}{8s}.
	$$
	Let us consider M\"obius transformations on the basis
	$$
	\phi_t(\lambda)=\frac{\lambda-\lambda_1}{\lambda_2-\lambda_1},
	$$
	and denote
	$$
	x(s):=\phi_t(\lambda_3)=\frac{\lambda_3-\lambda_1}{\lambda_2-\lambda_1},\quad y_0(s):=\phi_t(0)= \frac{\lambda_1}{\lambda_1-\lambda_2}.
	$$
	We calculate
	\begin{align*}
		&x(s)= \frac{3+24s-16 s^2-(4s-1) \sqrt{16 s^2-40 s+9}}
		{3+24s-16 s^2+(4s-1) \sqrt{16 s^2-40 s+9}},\\
		&y_0(s)= \frac{32s}{3+24s-16 s^2+(4s-1) \sqrt{16 s^2-40 s+9}},
	\end{align*}
	and
	$$
	\frac{dy_0}{dx}=\frac{dy_0}{ds}:\frac{dx}{ds}
	=
	-\frac{64 s^3-80 s^2-20 s+9-(3+16 s^2) \sqrt{16 s^2-40 s+9}}{24(4s-1)}
	$$
	Here $y_0$ is a Picard solution to the Painlev\'e VI $(0,0,0, 1/2)$ as a function of $x$.
	We apply the Okamoto transformation \eqref{Okamoto} and calculate:
	\begin{equation}\label{eq:Hitchin6a}
		\begin{aligned}
			y(x) &= y_0 + \frac{y_0(y_0-1)(y_0-x)}{x(x-1)\frac{dy_0}{dx}- y_0(y_0-1)}\\
			&=
			\frac{12-16 s}
			{
				3 + 24 s - 16 s^2
				+
				(4s-1) \sqrt{16 s^2-40 s+9}
			},\\
			x(s)&= \frac{3+24s-16 s^2-(4s-1) \sqrt{16 s^2-40 s+9}}
			{3+24s-16 s^2+(4s-1) \sqrt{16 s^2-40 s+9}}.
		\end{aligned}
	\end{equation}
	\begin{theorem}\label{th:painlevek6}
		The function $y(s)$, as a function of $x(s)$, given by the last two formulas \eqref{eq:Hitchin6a}, is a solution to Painlev\'e VI $(1/8,-1/8,1/8, 3/8)$ equation, which corresponds to a $6$-Poncelet inscribed family.
	\end{theorem}
	\begin{remark} Hitchin in \cite{Hitchin} obtained the following parameterization of a solution to Painlev\'e VI $(1/8,-1/8,1/8, 3/8)$ equation which corresponds to $6$-Poncelet polygons:
		\begin{equation}\label{eq:Hit6}
			\begin{aligned}
				x(S)&= \frac{S^3(S+2)}
				{2S+1},\\
				y(S)&= \frac{S(1+S+S^2)}{2S+1}.
			\end{aligned}
		\end{equation}
		{
			\begin{lemma}\label{lem:rep6}
				The solutions \eqref{eq:Hit6} are obtained from the solutions \eqref{eq:Hitchin6a} by the reparametrization
				$$s=\dfrac{2 S^2+5 S+2}{4 S}.$$
			\end{lemma}
		}
		
	\end{remark}

\

\

\subsection*{Acknowledgment}
We are grateful to Corinna Ulcigrai and the referees for valuable remarks and comments, which helped us improve the presentation.
The research  was supported
by the Australian Research Council, Discovery Project 190101838 \emph{Billiards within quadrics and beyond}, {  the Science Fund of Serbia grant \emph{Integrability and Extremal Problems in Mechanics, Geometry and
Combinatorics}, MEGIC, Grant No. 7744592  and the Simons Foundation grant no. 854861.}

\

\

{\bf Data availability.} Data sharing not applicable to this article as no datasets were generated
or analysed during the current study.

\

\end{document}